\documentclass[a4paper,notitlepage,twoside,leqno,12pt]{amsart}
\usepackage{amssymb}
\usepackage{mathrsfs}
\usepackage{titlesec}

\usepackage{epic,eepic}
\usepackage{bbm,pifont,latexsym}
\usepackage{anysize}
\marginsize{2.4cm}{2.4cm}{2.4cm}{2.4cm}
\usepackage{enumerate}
\usepackage{array}
\usepackage{cases}
\usepackage{amsthm}  
\usepackage[all]{xy} 
\usepackage{dsfont} 
\usepackage{enumerate}
\usepackage{graphicx} 
\usepackage{color}

\usepackage{dcolumn,indentfirst,color}
\usepackage[pdfpagemode=UseNone,pdfstartview=FitH]{hyperref}
\usepackage{amsmath,amssymb,amscd,amsthm,amsfonts,mathrsfs}
\usepackage{color,graphicx,xcolor,graphics}
\usepackage{tikz}
\usepackage{titlesec}
\usepackage{titletoc}

\newcommand{\wt}{\widetilde}

\newcommand{\mb}{\mathbb}
\newcommand{\mc}{\mathcal}

\newcommand{\tu}{\textup}
\newcommand{\ol}{\overline}

\newcommand{\tb}{\textbf}

\newcommand{\wh}{\widehat}
\newcommand{\wlC}{\widehat{{\mb{C}}}}

\newcommand{\HHH}{\mathcal{H}}

\newcommand{\G}{\Gamma}

\newcommand{\UUU}{\mathcal{U}}

\newtheorem{thm}{Theorem}[section]
\newtheorem{lem}[thm]{Lemma}

\newtheorem{prop}[thm]{Proposition}

\newtheorem{fact}[thm]{Fact}

\titleformat{\section}{\centering\normalsize}{\textsc{\thesection.}}{1em}{\textsc}
\titleformat{\subsection}{\normalsize}{\thesubsection.}{1em}{\textbf}
\titleformat{\subsubsection}{\normalsize}{\thesubsubsection.}{1em}{\textbf}

\title{On rational maps with buried critical points}

\author{Yan Gao}
\address{Yan Gao, School of Mathematics, Sichuan University, Chengdu 610064, China}
\email{gyan@scu.edu.cn}

\author{Luxian Yang}
\address{Luxian Yang, School of Mathematical Sciences, Zhejiang University, Hangzhou, 310027, China}
\email{yang\underline{ }luxian@163.com}

\author{Jinsong Zeng}
\address{Jinsong Zeng, School of Mathematics and Information Science, Guangzhou University, Guangzhou 510006, China}
\email{jinsongzeng@163.com}

\begin{document}
\maketitle
\begin{abstract}

In this paper, we construct geometrically finite rational maps with buried critical points on the boundaries of some hyperbolic components by using the pinching and plumbing deformations.

\end{abstract}

\section{Introduction}
Let ${\rm Rat}_d, d\geq 2,$ denote the space of rational maps of degree $d$ on the Riemann sphere $\wh{\mb{C}}$, with the topology that $f_n\to f$ if and only if $f_n$ uniformly converges to $f$ with respect to the spherical metric on $\wlC$. A
rational map is called \emph{hyperbolic} if all its critical points converge to attracting cycles under iterations.
The collection of hyperbolic rational maps forms an open subset of ${\rm Rat}_d$, whose connected components are called \emph{hyperbolic components}. A central conjecture in complex dynamics is

\noindent\emph{{\bf Hyperbolic Conjecture}: The hyperbolic rational maps are dense in ${\rm Rat}_d$.}

A related interesting question is to study the boundaries of hyperbolic components. In particular, one may ask what kinds of maps possibly lie on the boundary of a hyperbolic component?

Let $f$ be a rational map. A point $z\in\wlC$ is called \emph{buried} in the Julia set $\mathcal{J}_f$ if $z\in\mathcal{J}_f$ and $z$ is not on the boundary of any Fatou domain of $f$. In this paper, we provide a way to create rational maps with buried critical points on the boundaries of some hyperbolic components.

A rational map is called \emph{geometrically finite} if the orbit of every critical point is either finite or converges to an attracting or parabolic cycle. A geometrically finite rational map is \emph{subhyperbolic} if it has no parabolic cycles. We say a set (or a point) $E$ is \emph{preperiodic} under a map $f$ if and only if $f^{n+p}(E)=f^n(E)$ for some minimal integers $p\geq 1$ and $n\geq 1$.

\begin{thm}\label{maintheorem}
Let $\HHH$ be a hyperbolic component in ${\rm Rat}_d, d\geq 2$, which contains a map $R$ satisfying
\begin{itemize}
  \item the Julia set of $R$ is a Sierpi\'nski carpet;
  \item a preperiodic Fatou domain of $R$ contains a critical point.
\end{itemize}
Then the boundary of $\HHH$ contains a geometrically finite map with a buried critical point.
\end{thm}
It is known that the Julia set of a hyperbolic rational map is a Sierpi\'{n}ski carpet if and only if the boundaries of all the Fatou domains are Jordan curves and each pair of them are disjoint. There are many examples of hyperbolic rational maps whose Julia sets are Sierpi\'{n}ski carpets; see \cite{Dev13,DFGJ14,Mil93} etc. Conjecturally, the hyperbolic components containing those maps are relatively compact in $\tu{Rat}_d$ \cite[Question 5.3]{McM94} (established in degree two for maps which have two distinct cycles of Fatou domains of period at least two; see \cite{Eps00}).

In the case $d=2$, many hyperbolic components of $\tu{Rat}_2$ satisfy the conditions in Theorem \ref{maintheorem}; see \cite[Theorem A]{DFGJ14}.

The main tools used to prove Theorem \ref{maintheorem} are the pinching and plumbing deformations developed in \cite{CT18}. 

The outline of the proof is as follows. Let $R_0\in\HHH$ and $c_{R_0}$ be a critical point in a preperiodic Fatou domain. Starting from $R_0$, we find a pinching path $R_t, t\geq 0$ in $\HHH$ for which the distance between $c_{R_t}$ and $\mc{J}_{R_t}$ converges to 0 as $t$ tends to $\infty$, where $c_{R_t}$ is the critical point of $R_t$ deformed from $c_{R_0}$. The limit map $R_\infty:=\lim_{t\to\infty}R_t\in\partial \HHH$ has a unique parabolic cycle, and the limit $c_{R_\infty}:=\lim_{t\to\infty}{c_{R_t}}$ is a critical point of $R_\infty$ eventually falling into the parabolic cycle.  Then applying the plumbing surgery, we get another pinching path $g_t$ starting from a subhyperbolic map $g_0$ and terminating at $g_\infty=R_\infty$, satisfying that the actions of $g_t$ and $g_\infty$ on their Julia sets are topologically conjugate. This plumbing deformation splits the unique parabolic cycle of $g_\infty$ into an attracting cycle and a repelling cycle. Finally, we choose a pinching path $f_t$ starting from $f_0:=g_0$ to create a buried critical point for the limit map $f_\infty:=\lim_{t\to\infty}f_t$. Actually, the buried critical point is $c_{f_\infty}:=\lim_{t\to\infty}c_{f_t}$. The key point is to show that the two pinching paths $g_t$ and $f_t$ always stay on the boundary of $\HHH$.
\vskip 0.3cm
\emph{Acknowledgment} We sincerely thank Guizhen Cui and Yongcheng Yin for helpful discussions. The research is supported by the grants no.11801106 and no.11871354 of NSFC.

\section{Preliminary}
\subsection{Semi-rational maps and c-equivalent}
Let $f:\wh{\mb{C}}\to \wh{\mb{C}}$ be a branched covering. The \emph{orbit} of a set $E\subseteq\wh{\mb{C}}$ under $f$, denoted by $\tu{orb}(E)$, is the set $\cup_{i\geq 0}f^i(E)$. The \emph{postcritical set} $\tu{post}(f)$ of $f$ is defined as the closure of $\tu{orb}(\{\tu{all ciritival values of }f\})$. Following \cite{CT11}, $f$ is called \emph{semi-rational} if
\begin{itemize}
	\item[$\bullet$] the accumulation set of $\tu{post}(f)$, denoted by $\tu{post}(f)'$, is finite;
	\item[$\bullet$] $f$ is holomorphic in a neighborhood of $\tu{post}(f)'$;
	\item[$\bullet$] every cycle in $\tu{post}(f)'$ is attracting.
\end{itemize}
An open set $\mc{W}$ is said to be a \emph{fundamental set} of $f$ if $\mc{W}\subseteq f^{-1}(\mc{W})$ and $\mc{W}$ contains every attracting cycle in $\tu{post}(f)'$.

Two continuous maps $\phi$ and $\psi$  from a set $X \subseteq \wh{\mb{C}}$ to a set $Y\subseteq\wh{\mb{C}}$ are said to be \emph{homotopic} rel. a subset $A$ (maybe empty) of $X$ if there exists a continuous map $H:X\times [0,1]\to Y$ such that
$$H(x,0)=\phi(x), H(x,1)=\psi(x)\,\, \forall x\in X\tu{ and }H(x,t)=\phi(x)\,\, \forall x\in A~~ \forall t\in[0,1];$$
and \emph{isotopic} rel.\,$A$ if the map $H|_{X\times\{t\}}:X\to Y$ is a homeomorphism for each $t\in[0,1]$. In particular, a homotopy $H$ rel. $A$ is called a \emph{pseudo-isotopy} rel.\,$A$ if $H|_{X\times \{t\}}$ is a homeomorphism for each $t\in[0,1)$; and in this case the map $H|_{X\times{\{1\}}}$ is said to be the \emph{end} of the pseudo-isotopy $H$.

Two semi-rational maps $f_1$ and $f_2$ are called $\emph{c-equivalent}$, if there is a pair $(\phi, \psi)$ of homeomorphisms of $\wh{\mb{C}}$ and a fundamental set $\mc{W}$ of $f_1$ such that
\begin{itemize}
	\item $\phi\circ f_1=f_2\circ\psi$ on $\wlC$;
	\item $\phi$ is holomorphic in $\mc{W}$;
	\item $\phi$ and $\psi$ are isotopic relative to $\tu{post}(f_1)\cup \ol{\mc{W}}$.
\end{itemize}

The following is the rigidity part of Thurston Theorem; see \cite{CT11,DH93} for details.
\begin{thm}\label{theorem:thurston}
If two subhyperbolic rational maps are c-equivalent, then they are conformal conjugate.
\end{thm}

A continuous onto map $\eta:\wlC\to\wlC$ is called \emph{monotone} if both $\eta^{-1}(w)$ and $\wlC\setminus\eta^{-1}(w)$ are connected for each point $w\in\wh{\mb{C}}$.

\begin{lem}[from homotopy to isotopy]\label{lem:isotopy}
	Let $f_1$ and $f_2$ be two semi-rational maps. Let $\mc{W}_i$ be a fundamental set of $f_i$ consisting of finitely many Jordan domains with disjoint closures for each $i=1,2$. Suppose that $(\phi,\psi)$ is a pair of monotone maps on $\wlC$ fulfilling that
	\begin{enumerate}
		\item $\phi\circ f_1= f_2\circ \psi$;
		\item $\phi(x)=\psi(x)$\,\,\, $\forall x\in {\tu{post}(f_1)\cup\ol{\mc{W}_1}}$;
		\item $\phi:\tu{post}(f_1)\cup \ol{\mc{W}_1}\to\tu{post}(f_2)\cup\overline{\mc{W}_2}$ is a homeomorphism that is holomorphic in $\mc{W}_1$;
		\item $\phi^{-1}(\tu{post}(f_2)\cup\overline{\mc{W}_2})=\tu{post}(f_1)\cup \ol{\mc{W}_1}$;
		\item the restrictions $\phi, \psi: S_1\to S_2$ are homotopic rel. $\partial S_1$ with $S_i:=\wh{\mb{C}}\setminus(\tu{post}(f_i)\cup \ol{\mc{W}_i})$.
	\end{enumerate}
	Then $f_1$ and $f_2$ are c-equivalent.	
\end{lem}
\begin{proof}
Since $\phi$ is monotone, the map $\phi$ can be realized as the end of a pseudo-isotopy $\Phi:\wlC\times [0,1]\to \wlC$ rel. $\tu{post}(f_1)\cup \ol{\mc{W}_1}$, according to Moore's Theorem \cite{Mo25}. Let $h:=\Phi|_{\wlC\times \{0\}}$ be a homeomorphism.
Then by condition (1) and the homotopy lifting theorem, there exists a pseudo-isotopy $\Psi:\wlC\times [0,1]\to \wlC$ rel. $\tu{post}(f_1)\cup \ol{\mc{W}_1}$ such that
\begin{itemize}
\item  $\Psi|_{\wlC\times\{1\}}=\psi$ and $\wt{h}:=\Psi|_{\wlC\times \{0\}}$ is a homeomorphism;
\item $(\Phi|_{\wh{\mb{C}}\times\{t\}})\circ f_1=f_2\circ(\Psi|_{\wh{\mb{C}}\times\{t\}})$ on $\wlC$ for each $t\in[0,1]$, in particular, $h\circ f_1=f_2\circ \wt{h}$.
\end{itemize}
From condition (5), the restrictions $h,\wt{h}:S_1\to S_2$ are homotopic relative to $\partial S_1.$ By \cite[Theorem 1.12]{FM11}, they are isotopic rel. $\partial S_1$ on $S_1$. Since $\wt{h}=\psi=\phi=h$ on $\tu{post}(f_1)\cup \ol{\mc{W}_1}$ from condition (2), globally it holds that $h,\wt{h}:\wlC\to\wlC$ are isotopic rel. $\tu{post}(f_1)\cup \ol{\mc{W}_1}$. The proof of the lemma is complete.
\end{proof}

\subsection{Quasiconformal surgery in the attracting basins}

A standard quasiconformal surgery allows us to revise the dynamics of  rational maps on the attracting basins; see \cite{BF14,DH85}. We write a precise form in the following for the reference in Section \ref{proof}. The proof is left in the appendix; see Section \ref{app}.

\begin{lem}\label{lem:isotopy-annulus}
	Let $f$ be a rational map. Let $W_f$ be a preperiodic attracting Fatou domain with $f^{p+n}(W_f)=f^n(W_f)$ for some minimal integers $n\geq 1$ and $p\geq 1$. Assume the components of $\tu{orb}(W_f)$ are Jordan domains. Let $B_f$ be a quasi-disk compactly contained in $W_f$ such that for each $1\leq i\leq n+p$
	\begin{equation}\label{eq:deg}
	\tu{deg}(f^i:B_f\to f^i(B_f))=\tu{deg}(f^i:W_f\to f^i(W_f))\tu{ and } f^{n+p}(B_f)\Subset f^n(B_f).
	\end{equation}
	Let $R$ be a hyperbolic rational map whose Julia set is connected. Suppose that
	 there exists a Fatou domain $W_R$ of $R$ and a homeomorphism $\eta:\tu{orb}(\partial W_R)\to\tu{orb}(\partial W_f)$ such that
	 \begin{equation}\label{eq:com}
	\eta\circ R(z)=f\circ \eta(z) \text{ for all $z\in \tu{orb}(\partial W_R)$}.
	 \end{equation}
Then there exists a hyperbolic rational map $R_*$, a quasiconformal map $h$ on $\wlC$, and a pair of homeomorphisms  $\eta_0,\eta_1:\tu{orb}(\overline{W_{R_*}})\to\tu{orb}(\overline{W_f}),$ such that
	\begin{itemize}
\item the restriction $h:\wlC\setminus\UUU_{R}\to\wlC\setminus\UUU_{R_*}$ is a conjugacy between $R$ and $R_*$, where $\mc{U}_{R_*}$ is the grand orbit of the Fatou domain $W_{R_*}:=h(W_R)$ of $R_*$;
\item the restriction $h:\mc{F}_R\setminus\mc{U}_R\to \mc{F}_{R_*}\setminus\mc{U}_{R_*}$ is conformal, where $\mc{F}_R:=\wh{\mb{C}}\setminus\mc{J}_R$;
\item $\eta_0$ is isotopic to $\eta_1$ rel. $\tu{orb}(\ol{B_{R_*}})\cup \tu{orb}(\partial W_{R_*})$, where ${B}_{R_*}:=\eta_0^{-1}(B_f)\Subset W_{R_*}$ satisfies \eqref{eq:deg} for $R_*$.
\item the restriction $\eta_0:\tu{orb}(B_{R_*})\to \tu{orb}(B_R)$ is conformal;

\item  $\eta_0=\eta\circ h^{-1}$ on $\tu{orb}(\partial W_{R_*})$;
		\item $\eta_0\circ R_*=f\circ \eta_1$\text{ on $\tu{orb}(\overline{W_{R_*}})$}.		
	\end{itemize}
\end{lem}

\begin{lem}\label{lem:same-component}
Let $R_1$ and $R_2$ be two hyperbolic rational maps of degree $d\geq 2$ with connected Julia sets. If $\phi:\wlC\to\wlC$ is an orientation-preserving homeomorphism such that $\phi\circ R_1(z)=R_2\circ\phi(z)$ for all $z\in \mathcal{J}_{R_1}$, then $R_1$ and $R_2$ belong to the same hyperbolic component of $\tu{Rat}_d$.	
\end{lem}
\begin{proof}
We first consider the case that both $R_i$ are postcritically finite. Then the conjugacy $\phi:\mathcal{J}_{R_1}\to\mathcal{J}_{R_2}$ between $R_1|_{\mathcal{J}_{R_1}}$ and $R_2|_{\mathcal{J}_{R_2}}$ can be extended to a pair of homeomorphisms $(\phi_1,\phi_2)$ on $\wlC$ such that $\phi_1\circ R_1=R_2\circ \phi_2$. Moreover, $\phi_1$ and $\phi_2$ are isotopic rel. $\tu{post}(R_1)$ by Alexander's trick. Then there exists a M\"{o}bius transformation $\alpha$ such that $R_2=\alpha^{-1}\circ R_1\circ \alpha$ by Theorem \ref{theorem:thurston}.
Let $\alpha_t,1\leq t\leq 2,$ be a path in $\tu{SL}(2,\mathbb{C})$ with $\alpha_1:=\tu{id}$ and $\alpha_2:=\alpha$. Then $R_t:=\alpha_t^{-1}\circ R_1\circ \alpha_t$ is a path in ${\rm Rat}_d$ joining $R_1$ and $R_2$. Clearly, all the maps $R_t,1\leq t\leq 2,$ are hyperbolic. Thus $R_1$ and $R_2$ lie in the same hyperbolic component.

Otherwise, it is known from \cite[Theorem 9.3]{Mi12} that the hyperbolic component of $\tu{Rat}_d$ containing $R_i$ possesses a postcritically finite rational map $\wt{R}_i$ for each $i\in\{1,2\}$. The above arguments imply that $\wt{R}_1$ and $\wt{R}_2$ are in the same hyperbolic component, and so are $R_1$ and $R_2$. The proof of the lemma is complete.
\end{proof}
\subsection{Pinching and plumbing deformations}
Let $f$ be a subhyperbolic rational map. A family $\Gamma$ of finitely many disjoint open arcs is called \emph{admissible} for $f$, provided that
\begin{itemize}
	\item invariant: all arcs in $\Gamma$ avoid the critical points of $f$ and $f:\cup_{\gamma\in\Gamma}\gamma\to\cup_{\gamma\in\Gamma}\gamma$ is a homeomorphism;
	\item each $\gamma\in\Gamma$ lies in a geometrically attracting (not super-attracting) periodic Fatou domain and it joins the attracting periodic point in this basin to a point in the boundary (which must be a periodic repelling point in the Julia set by the above invariant property);
	\item non-separating: the set $\mc{J}_f\setminus f^{-i}(\cup_{\gamma\in\Gamma}\ol{\gamma})$ is connected for each $i\geq 0$.
\end{itemize}

In \cite{CT18}, the authors proved that, given a subhyperbolic rational map $f$ and an admissible family $\Gamma$ for $f$, one can shrink the iterated pre-images of arcs in $\Gamma$ by occupying a special quasiconformal deformation $f_t=\phi_t\circ f\circ \phi_t^{-1}, t\geq 0$ with $f_0=f$. The deformation takes place in a totally $f$-invariant open subset of the Fatou set $\mc{F}_f:=\wh{\mb{C}}\setminus\mc{J}_f$. The set contains the iterated pre-images of arcs in $\Gamma$. Such a deformation $f_t, t\geq 0$ in $\tu{Rat}_d$ is called a \emph{pinching path supported on $\Gamma$}.
\begin{thm}[{\cite[Theorem 1.5]{CT18}}]\label{thm:pinching}
	Any pinching path $f_t=\phi_t\circ f\circ \phi_t^{-1}, t\geq 0$ supported on an admissible family $\Gamma$ has the following properties:
	\begin{itemize}
		
		\item $f_t$ converges uniformly to a geometrically finite rational map $f_\infty$ as $t\to\infty$;
		\item $\phi_t$ converges uniformly to a continuous onto map $\phi_\infty$ of $\wh{\mb{C}}$ as $t\to\infty$;
		\item $\phi_\infty\circ f(z)=f_\infty\circ \phi_\infty(z)$ for all $z\in \wh{\mb{C}}$;
		\item the map $\phi_\infty$ can be charactered explicitly that, for a point $x\in\widehat{\mathbb{C}}$,
		$$\#\phi_\infty^{-1}(x)>1\Leftrightarrow \phi^{-1}(x)\textup{ is a component of }\cup_{i\geq 0} f^{-i}(\cup_{\gamma\in\Gamma}\overline{\gamma});$$
		thus the restriction $\phi_\infty:\mc{J}_f\to\mc{J}_g$ is a semi-conjugacy;
		\item $\phi_t,t\ge0$ and $\phi_\infty$ are holomorphic in the open set $\mc{F}_f\setminus \mc{U}_{\Gamma}$, where $\mc{U}_{\Gamma}$ denotes the union of the Fatou domains which intersect the iterated preimages of arcs in $\Gamma$.
	\end{itemize}
\end{thm}
We have two remarks: first, the arcs in an admissible family are disjoint from the postcritical set in the original setting of \cite{CT18}, however, their arguments still work in our situation; second, the starting map $f_0$ of a pinching path is allowed to have parabolic cycles, while throughout this paper, we only deal with the pinching paths starting from subhyperbolic rational maps.

Conversely, a geometrically finite rational map with parabolic cycles is the limit of certain pinching paths. These possible pinching paths can be encoded by a finite set of combinatorial data, namely, plumbing combinatorics.

\begin{thm}[{\cite[Compare Theorem 1.6]{CT18}}]\label{thm:plumbing}
	Let $f_\infty$ be a geometrically finite rational map with parabolic cycles and let $\sigma$ be a plumbing combinatoric of $f_\infty$. Then $f_\infty$ is the limit of a pinching path $f_t=\phi_t\circ f\circ \phi_t^{-1}, t\geq 0$ along $\sigma$ that starts from a subhyperbolic rational map $f$ and supports on an admissible family $\Gamma$ of $f$.
\end{thm}
\section{Proof of the Main Theorem}\label{proof}

Let $\HHH$ be a hyperbolic component in Theorem \ref{maintheorem}. By the conditions, we can admit the following settings for some hyperbolic rational map $R\in\mc{H}$:
\begin{itemize}
	\item $U_{R,0},\cdots, U_{R,p-1}$ is a Fatou cycle of period $p\geq 1$;
	\item $z_R$ is the geometrically attracting periodic point in $U_{R,0}$;
	\item there is a minimal integer $n\geq0$ such that $R^{-(n+1)}(\cup_{i=0}^{p-1}U_{R,i})\setminus \cup_{i=0}^{p-1}U_{R,i}$ contains a critical Fatou domain $W'_R$;
	\item assume $W_R:=R(W_R')$ and $R^n(W_R)$ is the Fatou domain $U_{R,0}$;
	\item assume further that $c_R$ is the unique critical point in $W_R'$ and in the backward orbits $\cup_{i\geq 0}R^{-i}(\{c_R, z_R\})$; this can be done by a standard quasiconformal surgery \cite{BF14} and Lemma \ref{lem:same-component}.
\end{itemize}
	
In what follows, starting from $R$, we will construct a geometrically finite rational map in $\partial \HHH$ such that its Julia set possesses a buried critical point. The proof is broken up into four steps.

\subsection*{Step 1: Constructing an admissible family of arcs in $U_{R,0},\ldots,U_{R,{p-1}}$.}
We choose a \emph{fundamental annulus} $A$ near $z_R$ with the outer boundary $\gamma_+$ and inner boundary $\gamma_-$, i.e., $R|_{\overline{A}}$ is injective and $\overline{A}\cap R^p(\overline{A})=\gamma_-$. We require further that each of $\gamma_+$ and $\gamma_-$ contains a unique point in the orbit of $c_R$.
Take an arc $\tau_0:[0,1]\to\overline{A}$ such that
 \begin{itemize}
 \item $\tau_0(0)=\gamma_+\cap\tu{orb}(c_R),\tau_0(1)=\gamma_-\cap\tu{orb}(c_R)$;
 \item $\tau_0(0,1)\subseteq A$ and is disjoint from the orbits of the critical points of $R$.
 \end{itemize}
For every $k\geq1$, define $\tau_k:=R^{pk}(\tau_0)$. Since $A$ is a fundamental annulus, then $\tau_k$ is disjoint from $\tau_0,\ldots,\tau_{k-2}$ and $\tau_k\cap\tau_{k-1}=\tau_k(0)$.
Note that $\tau_0$ avoids the orbits of the critical points of $R^p$ in $U_{R,0}$, one can always lift $\tau_0$ in $U_{R,0}$ by the iterate $R^p$. Hence, for $k\geq1$, we inductively define $\tau_{-k}$ to be the lift of $\tau_{-(k-1)}$ by $R^p$ based at $\tau_{-(k-1)}(0)$.

Now we set $\gamma_R:=\cup_{k=-\infty}^{+\infty}\tau_k$.
Due to the expansion properties of $R$ (compare \cite[Theorem 18.11]{Mil11}), the arc $\gamma_R$ joins the point $z_R$ and a repelling periodic point $w_R\in\partial U_{R,0}$. Clearly $R^p$ sends $\gamma_R$ onto itself homeomorphically. Let $\Gamma$ be the collection of arcs $\gamma_R, R(\gamma_R),\ldots,R^{p-1}(\gamma_R)$. Then $\Gamma$ is an admissible family for $R$ in $U_{R,0}\cup \ldots\cup U_{R,p-1}$, satisfying that
\begin{itemize}
	\item the arcs in $\Gamma$ avoid the orbits of critical points of $R$ except that of $c_R\in W_R'$;
	\item the arcs in $\Gamma$ contain all the points of $\tu{orb}(c_R)$ in $U_{R,0}\cup \ldots\cup U_{R,p-1}$; this can be done, if we let $R^{n+1}(c_R)$ be close to $z_R$ by a quasiconformal surgery on $R^{n}(W_R')$.
\end{itemize}
Since $\mathcal{J}_R$ is a Sierpi\'nski carpet, each connected component $\xi$ of $\cup_{i\geq0}R^{-i}(\cup_{\gamma\in\Gamma}\overline{\gamma})$ can be explicitly charactered as follows.
\begin{fact}\label{fact:1}
The component $\xi$ is the closure of a component of $R^{-i}(\gamma)$ for some $i\geq 0$ and some $\gamma\in \Gamma$. Moreover, $\xi$ is a closed arc if and only if its orbit is disjoint from $c_R$, and is a star-like tree otherwise; in the latter case, the unique branched point is the iterated preimage of $c_R$ in $\xi$.
\end{fact}

\subsection*{Step 2: Pushing $c_R$ to the boundary of a Fatou domain}
Consider the pinching path $R_t=\phi_t\circ R\circ \phi_t^{-1}, t\geq 0$ in $\mathcal{H}$ supported on $\Gamma$. By Theorem \ref{thm:pinching}, the path $R_t$ converges uniformly to a geometrically finite rational map $g:=R_\infty\in\partial\HHH$ as $t\to\infty$, with every component of $\cup_{i\geq0}R^{-i}(\cup_{\gamma\in\Gamma}\overline{\gamma})$ collapsing to a point; $\phi_t$ converges uniformly to a continuous onto map $\phi_\infty$; and the labeled points $c_{R_t}:=\phi_t(c_R), z_{R_t}:=\phi_t(z_R), w_{R_t}:=\phi_t(w_R)$ converge to $c_g, z_g, w_g$, respectively. Together with Fact \ref{fact:1}, we get a precise description of the dynamics of $g$.
\begin{fact}\label{fact:2}
	The followings hold:
\begin{enumerate}
	\item the points $z_g$ and $w_g$ coincide and is a parabolic point of $g$ of period $p$; the critical point $c_g$ belongs to $\mathcal{J}_g$ with $g^{n+1}(c_g)=z_g$;
	\item  considering the restriction $\phi_\infty|_{\mathcal{J}_R}:\mathcal{J}_R\to \mathcal{J}_g$, for a point $x\in\mathcal{J}_g$,
	$\#\phi_\infty^{-1}(x)>1$ if and only if $g^i(x)=c_g$ for some $i\geq0$; and if $\#\phi_\infty^{-1}(x)>1$, then $(\phi_{\infty}|_{\mathcal{J}_R})^{-1}(x)$ is contained in the boundary of a Fatou domain $U$ of $R$ such that $R^i(W)=W_R'$ for some $i\geq0$;
\item each Fatou domain $U$ of $R$ corresponds to finitely many Fatou domains of $g$, say $$U_1(g),\ldots,U_s(g)\tu{ for some $s\geq 1$}$$ under the relation that $\phi_{\infty}(\overline{U})=\cup_{i=1}^s\overline{U_i(g)}$. Clearly by Fact \ref{fact:1} when $\overline{U_k(g)}\cap \overline{U_{\ell}(g)}\not=\emptyset$ for some $k\not=\ell$, this intersection is a singleton and is a iterated preimage of $c_g$ under $g$.
\item every Fatou domain of $g$ is a Jordan domain and $\phi_\infty(\overline{U})\cap \phi_{\infty}(\overline{U'})=\emptyset$ for a pair of distinct Fatou domains $U$ and $U'$ of $R$.
\end{enumerate}
\end{fact}

\subsection*{Step 3: Plumbing to get a subhyperbolic rational map}

We want to perturb $g$ to a subhyperbolic rational map without changing its dynamics on the Julia set. By Theorem \ref{thm:plumbing}, this can be realized by a plumbing surgery. Precisely,

 \begin{fact}\label{fact3}
 There exists a subhyperbolic rational map $f$ and an admissible family $\Gamma_f=\{\gamma_f,\ldots, f^{p-1}(\gamma_f)\}$ with $\gamma_f\cap\tu{post}(f)=\emptyset$ in a Fatou cycle $U_{f,0},\ldots,U_{f,p-1}$ of period $p$, such that the pinching path
$$g_t=\psi_t\circ f\circ \psi_t^{-1}, t\geq 0$$
supported on $\Gamma_f$ satisfies
\begin{enumerate}
\item $g_t\to g,\ \psi_t\to\psi_\infty$ as $t\to\infty$;
\item each component of $\cup_{i\geq0}f^{-i}(\cup_{\gamma\in\Gamma_f}\overline{\gamma})$ intersects $\mathcal{J}_f$ in exactly one point, thus
$\psi_\infty:\mathcal{J}_f\to \mathcal{J}_g$ is a conjugacy between $f$ and $g$;
\item the closure of each Fatou domain of $f$ is one-to-one corresponding to that of $g$ under the map $\psi_\infty$;
\item the point $c_f:=(\psi_\infty|_{\mc{J}_f})^{-1}(c_g)$ is the unique critical point of $f$ in $\mathcal{J}_f$ and $w_f:=f^{n+1}(c_f)$, the endpoint of $\gamma_f\subseteq U_{f,0}$ in $\mathcal{J}_f$,  is a repelling point of period $p$.
\end{enumerate}
\end{fact}

The key point for the proof of Theorem \ref{maintheorem} is to show that the pinching path $g_t$ belongs to $\partial \mathcal{H}$. Since one can regard an arbitrary $g_{t_0}$ as the initial map $f$ of the pinching path, it reduces to prove the following.

\begin{prop}\label{prop:boundary}
	$f\in \partial \mathcal{H}$.
\end{prop}
\begin{proof}
	The strategy is as follows: we first disturb $f$ to a quasi-regular map $F_r$; and then prove that $F_r$ is c-equivalent to a normalized rational map $R_r\in \mathcal{H}$; finally a known result implies that $R_r$ tends to $f$ as $r\to 0$.
	

Let $W_f$ be the unique Fatou domain of $f$ whose boundary contains the critical value $v_f:=f(c_f)$ and $W'_f$ be the union of the $\tau:=\tu{deg}(f,c_f)$ components of $f^{-1}(W_f)$ whose boundaries possess $c_f$. Now we choose a Jordan disk $D_r$ around $v_f$ such that
	\begin{itemize}
		\item $\textup{diam }D_r<r$;
		\item $\partial D_r\setminus\overline{W_f}$ is an open arc buried in $\mathcal{J}_f$, i.e., it is disjoint from the boundary of each Fatou domain.
	\end{itemize}
	Let $\xi_r:\widehat{\mathbb{C}}\to\widehat{\mathbb{C}}$ be a quasiconformal map which is identity outside of $D_r$ and sends $v_f$ to a point, namely $v_{r}$, within $D_r\cap W_f$. Then the quasi-regular map $F_r$ is defined as
	$$F_r:=\xi_r\circ f:\widehat{\mathbb{C}}\to\widehat{\mathbb{C}}.$$
	Clearly $F_r$ coincides with $f$ except on $D'_r$, which is the component of $f^{-1}(D_r)$ containing $c_f$. Hence $F_r$ is semi-rational.
\begin{lem}\label{lem:equivalent}
		The map $F_r$ is c-equivalent a hyperbolic rational map $R_r$ in $\mc{H}$.
	\end{lem}
\begin{proof}
We have three rational maps $R, g$ and $ f$. Their dynamics are related by Theorem \ref{thm:pinching}, Facts \ref{fact:2} and \ref{fact3}, with $g$ the intermediate rational map:
\[R\overset{(R_t, \phi_t)}{\longrightarrow}g\overset{(g_t,\psi_t)}{\longleftarrow}f.\]
Recall that $W_R'$ is the Fatou domain of $R$ containing $c_R$ and  $W_R=R(W_R')$. By Fact \ref{fact:2} (2) and Fact \ref{fact3} (3), it holds that $\phi_\infty(\ol{W_R})=\ol{W_g}=\psi_\infty(\ol{W_f})$.

For each $\delta\in\{R,f,g\}$, let $\UUU_\delta$ denote the grand orbit of $W_\delta$, i.e., the union of Fatou domains $U$ of $\delta$ such that $\delta^i(U)=\delta^j(W_\delta)$ for some $i$ and $j$. For simplicity of the statement, we assume that
\begin{equation}\label{eq:assumption1}
\emph{$c_R$ is the unique critical point in the preperiodic Fatou domains of $\UUU_R$.}
\end{equation}
By Theorem \ref{thm:pinching}, Fact \ref{fact:2} (2) and Fact \ref{fact3} (2), we conclude that
 \begin{fact}\label{fact:4}
 The followings hold:
 \begin{enumerate}
\item  $\eta:=\psi_\infty^{-1}\circ\phi_\infty:\widehat{\mathbb{C}}\setminus\mathcal{U}_R\to \widehat{\mathbb{C}}\setminus\mathcal{U}_f$ is a semi-conjugacy between $R$ and $f$; 
\item the restriction $\eta:\mc{F}_R\setminus \UUU_R\to\mc{F}_f\setminus\UUU_f$ is a conformal isomorphism;

\item If the orbit of a Fatou domain $U_R\subseteq \mc{U}_R$ avoids $W_R'$, then $\eta(\partial U_R)$ is the boundary of a Fatou domain in $\UUU_f$; otherwise, $\eta(\partial U_R)$ is the union of boundaries of $\tau$ Fatou domains of $f$ in $\UUU_f$.
  \end{enumerate}
\end{fact}

 Choose a quasi-disk  $B_f$ compactly contained in $W_f$ such that $f^i(B_f)(=F^i_r(B_f))$, $0\leq i\leq n+p$, are quasi-disks satisfying
 \begin{equation}\notag
 \overline{f^{n+p}(B_f)}\subseteq f^{n}(B_f).
 \end{equation}
 Assume further that $B_f$ is so large that $\tu{orb}(B_f)$ covers the set $$\tu{post}(F_r)\cap \tu{orb}(W_f).$$

Since $\eta:\tu{orb}(\partial W_R)\to\tu{orb}(\partial W_f)$ is a conjugacy between $R$ and $f$ by Fact \ref{fact:4} (1)(3), we can apply Lemma \ref{lem:isotopy-annulus} to combine the two sub-dynamics: $f$ on $\tu{orb}(W_f)$ and $R$ on $\wh{\mb{C}}\setminus \mc{U}_R$ to produce a new rational map $R_r$, where $r$ measures the position of the critical value of $R_r$ in $W_{R_r}$. According to Lemma \ref{lem:same-component}, $R_r\in\mc{H}$. Precisely, there is a rational map $R_r\in\HHH$, a quasiconformal map $h:\wlC\to\wlC$ and two homeomorphisms $\eta_0,\eta_1:\tu{orb}(\overline{W_{R_r}})\to\tu{orb}(\overline{W_f})$, such that
\begin{fact}\label{fact6}
	 The followings hold:
\begin{enumerate}
	\item the restriction $h:\wlC\setminus\UUU_{R}\to\wlC\setminus\UUU_{R_r}$ is a conjugacy between $R$ and $R_r$, where $\mc{U}_{R_r}$ is the grand orbit of the Fatou domain $W_{R_r}:=h(W_R)$ of $R_r$;
	\item the restriction $h:\mc{F}_R\setminus\mc{U}_R\to \mc{F}_{R_r}\setminus\mc{U}_{R_r}$ is conformal;
\item $\eta_0$ is isotopic to $\eta_1$ rel. $\tu{orb}(\ol{B_{R_r}})\cup \tu{orb}(\partial W_{R_r})$, where ${B}_{R_r}:=\eta_0^{-1}(B_f)\Subset W_{R_r}$ satisfying \eqref{eq:deg} for $R_r$;
\item the restriction $\eta_0:\tu{orb}(B_{R_r})\to \tu{orb}(B_R)$ is conformal;

\item  $\eta_0=\eta\circ h^{-1}$ on $\tu{orb}(\partial W_{R_r})$;
\item $\eta_0\circ R_r=f\circ \eta_1$\text{ on $\tu{orb}(\overline{W_{R_r}})$};
	\end{enumerate}
moreover, by a standard quasiconformal surgery on the Fatou domain $W_{R_r}':=h(W_R')$ if necessary, we can also assume

\ \ (7) $c_{R_r}$ is the unique critical point in $W_{R_r}'$ and $v_{R_r}:=R_r(c_{R_r})\in W_{R_r}$ is mapped by $\eta_0$ to the critical value $v_r$ of $F_r$.
\end{fact}
In the following, we will continuously extend $\eta_0$ and $\eta_1$ to  $\wlC$ such that $\eta_0\circ R_r=F_r\circ \eta_1$ holds on $\wlC$. Firstly, by Fact \ref{fact:4} (1) and Fact \ref{fact6} (1)(5), one has a continuous extension $$\eta_0:(\wlC\setminus \UUU_{R_r})\cup\tu{orb}(\overline{W_{R_r}})\to (\wlC\setminus \UUU_{f})\cup\tu{orb}(\overline{W_{f}})$$ by setting
$\eta_0(z)=\eta\circ h^{-1}(z)\tu{ for all $z\in \wlC\setminus \UUU_{R_r}$}.$
According to Fact \ref{fact:4} and Fact \ref{fact6} (1)(2), we conclude that
 \begin{fact}\label{fact:5}
 	The followings hold:
 \begin{enumerate}
 \item  the restriction $\eta_0:\wh{\mb{C}}\setminus\mc{U}_{R_r}\to \wh{\mb{C}}\setminus\mc{U}_f$ is a semi-conjugacy between $R_r$ and $f$;
\item the restriction $\eta_0:\mc{F}_{R_r}\setminus \UUU_{R_r}\to\mc{F}_f\setminus\UUU_f$ is conformal;
 \item if the orbit of a Fatou domain $U_{R_r}$ of $R_r$ avoids $W_{R_r}'$, then $\eta_0(\partial U_{R_r})$ is the boundary of a Fatou domain of $f$; otherwise, $\eta_0(\partial U_{R_r})$ is the union of the boundaries of $\tau$ Fatou domains in $\UUU_f$.
  \end{enumerate}
\end{fact}

It remains to extend $\eta_0$ to each component $U_{R_r}$ of $\UUU_{R_r}\setminus\tu{orb}(W_{R_r})$. There are three cases:
\begin{itemize}
\item The orbit of $U_{R_r}$ is disjoint from $W_{R_r}'$. In this case, there exists a minimal $k\geq 1$ such that $R_r^k(U_{R_r})\subseteq \tu{orb}(W_{R_r})$. By Fact \ref{fact:5} (3), $\eta_0(\partial U_{R_r})$ bounds a component $U_f$ of $\mc{U}_f$. Since $R_r^k|_{U_{R_r}}$ and $f^k|_{U_f}$ are homeomorphisms by assumption \eqref{eq:assumption1}, we define $\eta_0|_{U_{R_r}}:=(f^k|_{U_f})^{-1}\circ\eta_0\circ R_r^k|_{U_{R_r}}$. By Fact \ref{fact:5} (1)(3) this extension is continuous.

\item $U_{R_r}=W_{R_r}'$.  Recall that $W_f'$ denotes the union of $\tau$ components of $f^{-1}(W_f)$ whose boundaries contain $c_f$, and $D_r$ a small disk containing $v_f$ chosen at the beginning of Proposition \ref{prop:boundary}. Fact \ref{fact:5} (3) implies $\eta_0(\partial W_{R_r}')=\partial W_f'$, and $\eta_0^{-1}(\partial D_r)$ is a Jordan curve surrounding a disk, denoted by $\Delta_r$, such that $v_{R_r}\in\Delta_r$.  Choose a closed arc $\beta$ such that $$\beta(0,1)\subseteq W_{R_r}\cap \Delta_r, \beta(0)=v_{R_r}\tu{ and }\beta(1)=\eta_0^{-1}(v_f)\in\Delta_r\cap\partial W_{R_r}.$$ Let $\beta'$ be the component of $R_r^{-1}(\beta)$ containing $c_{R_r}$. Then $\beta'$ divides $W_{R_r}'$ into $\tau$ disks. We continuously extend $\eta_0$ from $\partial W_{R_r}'$ into $W'_{R_r}$ satisfying that $\eta_0(\beta')=c_f$ and $\eta_0$ sends the $\tau$ components of $W_{R_r}'\setminus \beta'$ homeomorphically onto those of $W_f'$.

\item There exists a minimal $k\geq 1$ such that $R_r^k(U_{R_r})=W_{R_r}'$. Let $U_f$ be the union of the $\tau$ Fatou domains of $f$ bounded by $\eta_0(\partial U_{R_r})$. In this case, by assumption (\ref{eq:assumption1}), we let $\eta_0|_{U_{R_r}}:=(f^{k}|_{\ol{U_f}})^{-1}\circ \eta_0\circ R_r^k|_{U_{R_r}}$.
\end{itemize}
Thus, we obtain a continuous onto map $\eta_0:\wlC\to \wlC$ satisfying that for each $z\in\wlC$
$$\#\eta_0^{-1}(z)>1\Leftrightarrow \exists\, i\geq 0\tu{ s.t.}f^i(z)=c_f\Leftrightarrow \eta_0^{-1}(z)\tu{ is a component of }\cup_{i\geq 0}R_r^{-i}(\beta').$$
Thus $\eta_0$ is monotone.

We now extend $\eta_1$ to $\wlC$. Recall that $D_r'$ is the component of $f^{-1}(D_r)$ containing $c_r$. Assume $\Delta_r'$ is the component of $R_r^{-1}(\Delta_r)$ containing $c_{R_r}$. Then $\overline{\Delta_r'}=\eta_0^{-1}(\overline{D_r'})$ by the choice of $D_r$ and the construction of $\eta_0$. For each point $z\in\wlC\setminus \Delta_r'$, define
\[\eta_1(z):=\left\{
    \begin{array}{ll}
      \eta_1(z), & \hbox{if $z\in\tu{orb}(W_{R_r})$;} \\
      \eta_0(z), & \hbox{otherwise.}
    \end{array}
  \right.
\]
By the construction of $\eta_0$, we have $\eta_0\circ R_r(z)=f\circ \eta_1(z)$ on $\wlC\setminus \Delta_r'$. Since $F_r=f$ on $\wlC\setminus D_r'$, it follows that $\eta_0\circ R_r(z)=F_r\circ \eta_1(z)$ on $\wlC\setminus \Delta_r'$.

The extension of $\eta_1:\Delta_r'\setminus\{c_{R_r}\}\to D'_r\setminus\{c_f\}$ is defined by the following lift:
	\[
\begin{CD}
{\Delta_r'\setminus\{c_{R_r}\}}@> {\eta_1} >> {D_r'\setminus\{c_f\}}\\
@ V{R_r} VV @ VV {F_r}V\\
{\Delta_r\setminus\{v_{R_r}\}}@> {\eta_0} >> {D_r\setminus\{v_r\}}.\\
\end{CD}
\]
Finally, we let $\eta_1(c_{R_r})=c_f$. Then $\eta_1$ is a monotone map satisfying $\eta_0\circ R_r=F_r\circ \eta_1$ on $\wlC$.

 By the constructions of $\eta_0$ and $\eta_1$, we see that the restrictions
 $\eta_0,\eta_1: S_{R_r}\to S_f$ are homotopic rel. $\partial S_{R_r}$, where $$S_\delta:=\wh{\mb{C}}\setminus (\tu{post}(\delta)\cup \ol{\tu{orb}(B_{\delta})})$$ for each $\delta\in\{R_r,f\}$,
 and $$\eta_0=\eta_1:\tu{post}(R_r)\cup\ol{\tu{orb}(B_{R_r})}\to \tu{post}(f)\cup\ol{\tu{orb}(B_f)}$$ is a homeomorphism that is holomorphic in $\tu{orb}(B_{R_r})$. It then follows from Lemma \ref{lem:isotopy} that $R_r$ is c-equivalent to $F_r$. We complete the proof of Lemma \ref{lem:equivalent}.
\end{proof}	

We continue to prove Proposition \ref{prop:boundary}. Choose three distinct points $a_1,a_2,a_3\in \tu{post}(f)\cap \mathcal{F}_f$. According to Lemma \ref{lem:equivalent}, for each $r>0$, there exists a pair of homeomorphism $(\phi_{0,r},\phi_{1,r})$ by which $F_r$ is c-equivalent to $R_r$. We normalize $\phi_{0,r}$ such that it fixes $a_1,a_2,a_3$. Lemma \ref{lem:same-component} implies that such normalized rational maps $R_r$ still belong to $\HHH$. By \cite[Proposition 3.3]{Gao19} or by similar arguments in \cite[Step IV of the proof of Theorem 1.1]{GZ18}, the rational maps $R_r$ uniformly converge to $f$ as $r\to 0$, which implies $f\in\partial \HHH$.
\end{proof}

\subsection*{Step 4: Create a buried critical point}
Recall that $U_{R,0},\ldots,U_{R,p-1}$ is the attracting Fatou cycle of $R$, where the pinching deformation in Step 1 takes place. The map $\eta$, defined in Fact \ref{fact:4} (1), sends the boundary of this Fatou cycle onto that of an attracting Fatou cycle of $f$, say $U_{f,0},\ldots,U_{f,p-1}$. By a similar method as stated in Step 1, one can get a pair of
 open arcs $\{\alpha,\alpha'\}$ in $U_{f,0}$ such that
\begin{itemize}
	\item $f^p:(\alpha,\alpha')\to(\alpha', \alpha)$ are homeomorphisms;
	\item $\alpha\cap \alpha'=\emptyset$ and $(\alpha\cup\alpha')\cap \tu{post}(f)=\emptyset$;
	\item $\alpha(0)=\alpha'(0)$ is the attracting periodic point in $U_{f,0}$, and $\alpha(1)\not=\alpha'(1)\in\partial U_{f,0}$ lie in a repelling cycle of period $2p$.
\end{itemize}

It is clear that $\wt{\Gamma}:=\{f^i(\alpha),f^i(\alpha'):0\leq i\leq p-1\}$ forms an admissible family of $f$. By Theorem \ref{thm:pinching}, we obtain a pinching path $f_t:=\zeta_t\circ f\circ\zeta_t^{-1}, t\geq 0$ supported on $\wt{\G}$ converging to a geometrically finite rational map $f_\infty$, such that
\begin{enumerate}
\item the set $\mathcal{O}=\{\zeta_\infty(f^i(\overline{\alpha\cup\alpha'})):0\leq i\leq p-1\}$ forms a unique parabolic cycle of $f_\infty$, with period $p$ and multiplier $-1$;
\item if a Fatou domain $U$ of $f$ is not in $\UUU_f$, then $\zeta_\infty(U)$ is an eventually attracting Fatou domain of $f_\infty$; otherwise, the set $U\setminus \cup_{i\geq 0}f^{-i}(\ol{\alpha\cup\alpha'})$ consists of infinitely many components, and these components are sent by $\zeta_\infty$ bijectively onto Fatou domains of $f_\infty$, which are eventually iterated under $f_\infty$ into the unique parabolic Fatou cycle of period $2p$;
\item $\zeta_\infty(\overline{U})\cap \zeta_\infty(\overline{U_{f,0}})=\emptyset$ for a Fatou domain $U\neq U_{f,0}$ of $f$.
\end{enumerate}

Note that $c_{f_\infty}:=\zeta_\infty(c_f)$ is the unique  critical point of $f_\infty$ in $\mc{J}_{f_\infty}$, and $w_{f_\infty}:=f^{n+1}_\infty(c_{f_\infty})
\in\zeta_\infty(\partial U_{f,0})$ is a repelling point of $f_\infty$ of period $p$. By statements (2)(3) above, the point  $w_{f_\infty}$ does not belong to the boundary of any Fatou domain of $f_\infty$, and so is $c_{f_\infty}$. Applying the same arguments in Proposition \ref{prop:boundary} to each subhyperbolic $f_t$, we see that the pinching path $f_t, t>0$ is contained in $\partial \mc{H}$. Thus the limit $f_\infty\in\partial \mc{H}$. Therefore, we obtain a geometrically finite map in $\partial \HHH$ with the buried critical point $c_{f_\infty}$. The proof of the main theorem is completed.
\section{Appendix}\label{app}
\begin{proof}[Proof of Lemma \ref{lem:isotopy-annulus}]
	Let $B_R$ be a quasi-disk in $W_R$ satisfying \eqref{eq:deg} for $R$. For each $\delta\in\{f,R\}$, we denote by $$\tb{W}_\delta=\cup_{i=1}^{n+p-1}\delta^i(W_\delta), \wt{\tb{W}}_\delta=\cup_{i=0}^{n+p-1} \delta^i(W_\delta)\tu{ and }\wt{\tb{B}}_\delta=(\delta|_{\wt{\tb{W}}_\delta})^{-1}(\cup_{i=1}^{n+p-1}\delta^i(B_\delta)).$$
	
	Take a conformal isomorphism $\alpha_1: \wt{\tb{B}}_f\to \wt{\tb{B}}_R$ which sends the component of $\wt{\tb{B}}_f$ in $f^i(W_f)$ onto that of $\wt{\tb{B}}_R$ in $R^i(W_R)$ for each $0\leq i\leq n+p-1.$ Let $\alpha_0: \tb{B}_f\to\tb{B}_R$ be the restriction of $\alpha_1$ on $\tb{B}_f:=f(\wt{\tb{B}}_f)\subseteq\wt{\tb{B}}_f$ and $\tb{B}_R:=\alpha_1(\tb{B}_f)$. We then introduce a holomorphic map $\mc{R}:\wt{\tb{B}}_R\to \tb{B}_R$ induced by the following diagram:
	\[
	\begin{CD}
	{\wt{\tb{B}}_f}@> {\alpha_1} >> {\wt{\tb{B}}_R}\\
	@ V{f} VV @ VV {\mc{R}}V\\
	{\tb{B}_f}@> {\alpha_0} >> {\tb{B}_R}.\\
	\end{CD}
	\]
	Let us set $\mc{R}|_{\wh{\mb{C}}\setminus \wt{\tb{W}}_R}=R|_{\wh{\mb{C}}\setminus \wt{\tb{W}}_R}$. In what follows, we will extend $\mc{R}$ as a quasi-regular map of $\wh{\mb{C}}$.
	
	For each $\delta\in\{f,R\}$, pick a quasi-circle $\gamma_\delta$ essentially contained in the annulus $$\mb{A}(\delta^n(\partial B_\delta),  \delta^{n+p}(\partial B_\delta)),$$ where $\mb{A}(\beta_1,\beta_2)$ denotes the annulus bounded by the Jordan curves $\beta_1$ and $\beta_2$. 
	From condition \eqref{eq:deg}, there exists a unique quasi-circle $\gamma_\delta^{i}$ in $\delta^{i}(W_\delta)$ such that $\delta^{n+p-i}(\gamma_\delta^i)=\gamma_\delta$ for each $1\leq i\leq n+p-1$. 
	When $\delta=f$, clearly $\gamma_\delta^i$ surrounds the disk $B_\delta^i:=\tb{B}_\delta\cap \delta^i(W_\delta)$, i.e., $B_\delta^i$ is contained in a complement component of $\gamma_\delta^i$; this still holds for $\delta=R$ if we choose the circle $\gamma_\delta$ sufficiently close to  $\delta^n(\partial B_\delta)$.
	
	Note that the set $\wt{\tb{W}}_\delta$ (resp. $\tb{W}_\delta$) can be decomposed into the three disjoint open sets
	$\wt{\tb{A}}_\delta, \wt{\tb{H}}_\delta, \wt{\tb{B}}_\delta$ (resp. ${\tb{A}}_\delta, {\tb{H}}_\delta, {\tb{B}}_\delta$), whose components are either disks or annuli, if we let
	$$\tb{A}_\delta:=\cup_{i}\mb{A}(\delta^i(\partial W_\delta),\gamma_\delta^i),~~~ \tb{H}_\delta:=\cup_{i}\mb{A}(\gamma_\delta^i, \partial B_\delta^i)=:\cup_{i} H_\delta^i,$$
	$\wt{\tb{A}}_\delta:=(\delta|_{\wt{\tb{W}}_\delta})^{-1}\tb{A}_\delta$ and $\wt{\tb{H}}_\delta:=\wt{\tb{W}}_\delta\setminus\ol{\wt{\tb{A}}_\delta\cup\wt{\tb{B}}_\delta}$, where $i$ runs over $[1,n+p-1]$.	
	
	Let $\alpha_0:\tb{A}_f\to \tb{A}_R$ be a homeomorphism such that its extension from each boundary $\gamma_f^i$ onto $\gamma_R^i$ is quasisymmetric. Let $\mc{R}|_{\wt{\tb{A}}_R}=R|_{\wt{\tb{A}}_R}$. Then from condition \eqref{eq:com}, the map $\alpha_1:\wt{\tb{A}}_f\to \wt{\tb{A}}_R$ can be defined by the lift:
	\[
	\begin{CD}
	{\wt{\tb{A}}_f}@> {\alpha_1} >> {\wt{\tb{A}}_R}\\
	@ V{f} VV @ VV {\mc{R}}V\\
	{\tb{A}_f}@> {\alpha_0} >> {\tb{A}_R.}\\
	\end{CD}
	\]
	Now we quasiconformally extend $\alpha_0:\partial \tb{H}_f\to\partial\tb{H}_R$ (resp. $\alpha_1:\partial\wt{\tb{H}}_f\to \partial\wt{\tb{H}}_R$) into the interior of $\tb{H}_f$ (resp. $\wt{\tb{H}}_f$). Then the quasi-regular map $\mc{R}:\wt{\tb{H}}_R\to\tb{H}_R$ is induced by the following diagram:
	\[
	\begin{CD}
	{\wt{\tb{H}}_f}@> {\alpha_1} >> {\wt{\tb{H}}_R}\\
	@ V{f} VV @ VV {\mc{R}}V\\
	{\tb{H}_f}@> {\alpha_0} >> {\tb{H}_R}.\\
	\end{CD}
	\]
	We also extend the map $\alpha_0$ over $W_f$ onto $W_R$ such that it is isotopic to the restriction $\alpha_1|_{W_f}$ relative to $B_f$.
	
	We see that the two maps $\alpha_0,\alpha_1:\wt{\tb{W}}_f\to \wt{\tb{W}}_R$ coincide in $\wt{\tb{B}}_f$ and statisfy $\mc{R}\circ \alpha_1=\alpha_0\circ f$ on $\wt{\tb{W}}_f$. However, they may not be isotopic rel. $\wt{\tb{B}}_f$. To fix this problem, in what follows, we will compose suitable twists around the annuli $H_R^i$ to the maps $\mc{R}$ and $\alpha_0$.
	
	For each $1\leq i\leq n+p-1$, let $\beta_i:H_R^i\to\mb{H}_{r_i}:=\{z: r_i<|z|<1\} $ be a conformal parameterization; the twist $T_i:R^i(W_R)\to R^i(W_R)$ is defined by
	\begin{equation}\notag
	T_i(z)=
	\left\{
	\begin{array}{ll}
	z &~~~\tu{ if }z\in R^i(W_R)\setminus H_R^i;\\
	\beta_i^{-1}(re^{\tb{i}(\theta+2\pi\frac{r-r_i}{1-r_i})})  &~~~\tu{ otherwise, }z\in H_R^i\tu{ and assume } \beta_i(z)=re^{\tb{i}\theta}\in\mb{H}_{r_i}.
	\end{array}
	\right.
	\end{equation}
	Then there exists a unique integer $n_i$ such that the map $T_i^{n_i}\circ (\alpha_0|_{R^i(W_R)})$ is isotopic to $\alpha_1|_{R^i(W_R)}$ relative to $R^i(B_f)$. Since the following commutative diagram holds:
	\[
	\begin{CD}
	{f^{i-1}(W_f)}@> {\alpha_1} >> {R^{i-1}(W_R)}\\
	@ V{f} VV @ VV {T_i^{n_i}\circ\mc{R}}V\\
	{f^i(W_f)}@> {T_i^{n_i}\circ \alpha_0} >> {R^i(W_R),}\\
	\end{CD}
	\]
	we may replace the behavior of $\alpha_0$ on $f^i(W_f)$ (resp. $\mc{R}$ on $R^{i-1}(W_R)$) by $T_i^{n_i}\circ \alpha_0$ (resp. $T_i^{n_i}\circ \mc{R}$). The new maps are still denoted by $\alpha_0$ and $\mc{R}$ respectively.
	
	Next we will construct a new quasiconformal structure $\mu$ on $\wh{\mb{C}}$ which is invariant under the quasi-regular map $\mc{R}:\wh{\mb{C}}\to\wh{\mb{C}}$. To do this, start from the standard conformal structure $\mu_0$ on the disks $\mc{R}^{i}(B_R)$ for $i=0,\cdots, n+p-1$, and also on all points of $\wh{\mb{C}}\setminus \mc{U}_R$, which are not in the iterated pre-images of $\wt{\tb{W}}_R$. Now we pull $\mu_0$ to the rest of $\wh{\mb{C}}$ under the action of $\mc{R}$ and its iterates. This will yield a well-defined quasiconformal structure $\mu$ on $\wh{\mb{C}}$, which has bounded dilatation since an orbit can pass through the union of annuli $H_R^i$ at most $n+p-1$ times. Using the measurable Riemann mapping theorem, we obtain a quasiconformal map $h$ sending $\mu$ to $\mu_0$. This implies that the map $R_*=h\circ \mc{R}\circ h^{-1}$ is a rational map.
	
	Finally, let $W_{R_*}=h(W_R)$ and $\eta_i=\alpha_i^{-1}\circ h^{-1}: \tu{orb}(\ol{W_{R_*}})\to \tu{orb}(\ol{W}_{f}), i\in\{0,1\}$. From the constructions above, the maps $R_*, \eta_0, \eta_1$ and $h$ are as required. The proof of the lemma is complete.
\end{proof}

\end{document}